\documentclass{amsart}
\usepackage{amsmath,amssymb, amscd}
\usepackage{amsfonts}

\newtheorem{theorem}{Theorem}[section]

\newtheorem{corollary}[theorem]{Corollary}

\newtheorem{definition}{Definition}
\newtheorem{example}[theorem]{Example}

\newtheorem{lemma}[theorem]{Lemma}

\newtheorem{proposition}[theorem]{Proposition}

\numberwithin{equation}{section}
\DeclareMathOperator{\Ima}

\begin{document}
\title{Weak c-ideals of a Lie algebra}
\author{David A. TOWERS}
\address{Lancaster University\\
Department of Mathematics and Statistics \\
LA$1$ $4$YF Lancaster\\
ENGLAND}
\email{d.towers@lancaster.ac.uk}
\author{Zekiye CILOGLU}
\address{Suleyman Demirel University\\
Department of Mathematics\\
$32260$, Isparta, TURKEY}
\email{zekiyeciloglu07@gmail.com}
\thanks{$2000$ \textit{Mathematics Subject Classification.} $17$B$05$, $17$B$%
20$, $17$B$30$, $17$B$50$. }

\begin{abstract}
A subalgebra $B$ of a Lie algebra $L$ is called a weak c-ideal of $L$ if
there is a subideal $C$ of $L$ such that $L=B+C$ and $B\cap C\leq B_{L}$
where $B_{L}$ is the largest ideal of $L$ contained in $B.$ This is
analogous to the concept of weakly c-normal subgroups, which has been
studied by a number of authors. We obtain some properties of weak c-ideals
and use them to give some characterisations of solvable and supersolvable
Lie algebras. We also note that one-dimensional weak c-ideals are c-ideals.
\end{abstract}

\maketitle

\section{\textbf{INTRODUCTION}}
Throughout $L$ will denote a finite-dimensional Lie algebra over a field $F$. 
If $B$ is a subalgebra of $L$ we define $B_L$, the {\em core} (with respect to $L$) of $B$ to be the largest ideal of $L$ contained in $B$. We say that a subalgebra $B$ of $L$ is a {\em weak c-ideal} of $L$ if there is a subideal $C$ of $L$ such that $L = B + C$ and $B \cap C \leq B_L$. This is a generalisation of the concept of a c-ideal which was studied in \cite{c-ideal}. It is analogous to the concept of weakly c-normal subgroup as introduced by Zhu, Guo and Shum in \cite{Zhu et al}; this concept has since been further studied by a number of authors, including Zhong and Yang (\cite{zhong}), Zhong, Yang, Ma and Lin (\cite{zhong et al}), Tashtoush (\cite{tashtoush}) and Jehad (\cite{Jehad}) who called them c-subnormal subgroups. 
\par
The maximal subalgebras of a Lie algebra $L$ and their relationship to the structure of $L$ have been studied extensively. It is well known that $L$ is nilpotent if and only if every maximal subalgebra of $L$ is an ideal of $L$ (see \cite{barnes}). A further result is that if $L$ is solvable then every maximal subalgebra of $L$ has codimension one in $L$ if and only if $L$ is supersolvable (see \cite{Barnes}). In \cite{c-ideal} similar characterisations of solvable and supersolvable Lie algebras were obtained in terms of c-ideals. The purpose here is to generalise these results to ones relating to weak c-ideals.
\par
In section two we give some basic properties of weak c-ideals; in particular, it is shown that weak c-ideals inside the Frattini subalgebra of a Lie algebra $L$ are necessarily ideals of $L$. In section three we first show that all maximal subalgebras of $L$ are c-ideals of $L$ if and only if $L$ is solvable and that $L$ has a solvable maximal subalgebra that is a c-ideal if and only if $L$ is solvable. Unlike the corresponding results for c-ideals, it is necessary to restrict the underlying field to characteristic zero, as is shown by an example. Finally we have that if all maximal nilpotent subalgebras of $L$ are weak c-ideals, or if all Cartan subalgebras of $L$ are c-ideals and $F$ has characteristic zero, then $L$ is solvable.
\par
In section four we show that if $L$ is a solvable Lie algebra over a general field and every maximal subalgebra of each maximal nilpotent subalgebra of $L$ is a weak c-ideal of $L$ then $L$ is supersolvable. If each of the maximal nilpotent subalgebras of $L$ has dimension at least two then the assumption of solvability can be removed. Similarly if the field has characteristic zero and $L$ is not three-dimensional simple then this restriction can be removed. In the final section we see that every one-dimensional subalgebra is a weak c-ideal if and only if it is a c-ideal.. 
\par

If $A$ and $B$ are subalgebras of $L$ for which $L = A + B$ and $A \cap B = 0$ we will write $L = A \oplus B$. The ideals $L^{(k)}$ and $L^k$ are defined inductively by $L^{(1)} = L^1 = L$, $L^{(k+1)} = [L^{(k)},L^{(k)}]$, $L^{k+1} = [L,L^k]$ for $k \geq 1$. If $A$ is a subalgebra of $L$, the {\em centralizer} of $A$ in $L$ is $C_{L}(A) = \{ x \in L : [x, A] = 0\}$.  
\section{\textbf{PRELIMINARY RESULTS}}

We
first give some basic properties of weak c-ideals.

\begin{definition}
\label{p:1} Let $I$ be a subalgebra of $L.$ We call $I$ a subideal of $L$ if
there is a chain of subalgebras%
\begin{equation*}
I=I_{0}<I_{1}<\ ...<I_{n}=L,
\end{equation*}%
where $I_{j}$ is an ideal of $I_{j+1}$ for each $0\leq j\leq n-1.$
\end{definition}

\begin{definition}
A subalgebra $B$ of a Lie algebra $L$ is a weak $c$-$ideal$ of $L$ if there
exists a subideal $C$ of $L$ such that%
\begin{equation*}
L=B+C\text{ and }B\cap C\leq B_{L},
\end{equation*}%
where $B_{L}$, the core of $B,$ is the largest ideal of $L$ contained in $B.$
\end{definition}

\begin{definition}
A Lie algebra $L$ is called weak\ c-simple if $L$ does not contain any weak\
c-ideals except the trivial subalgebra and $L$ itself.
\end{definition}

\begin{lemma}
Let $L$ be a Lie algebra. Then the following statements hold:

(1)\ Let $B$ be a subalgebra of $L.$ If $B$ is a $c$-$ideal$ of $L$ then $B$
is a weak $c$-ideal of $L.$

(2)\ $L$ is $weak\ c$-$simple$ if and only if $L$ is simple.

(3)\ If $B$ is a $weak\ c$-$ideal$ of $L$ and $K$ is a subalgebra with $%
B\leq K\leq L$, then $B$ is a $weak\ c$-$ideal$ of $K.$

(4)\ If $I$ is an ideal of $L$ and $I\leq B,$ then $B$ is a $weak\ c$-$ideal$
of $L$ if and only if $B/I$ is a $weak\ c$-$ideal$ of $L/I.$
\end{lemma}

\begin{proof}
(1)\ By the definition every ideal is a $c$-ideal and every $c$-ideal is a
weak\ c-ideal so the proof is obvious.

(2) Suppose first that $L$ is simple and let $B$ be a weak\ c-ideal with 
$B\neq L$. Then 
\begin{equation*}
L=B+C\text{ and }B\cap C\leq B_{L}
\end{equation*}
where $C$ is a subideal of $L.$ But, since $L$ is simple, $B_{L}$ must be $%
0.\ $Moreover, $C\neq 0$ so $C=L$. Hence $B=0$ and $L$ is weak $c$-simple.

Conversely, suppose $L$ is weak\ c-simple. Then, since every ideal of $L$
is a weak\ c-ideal, $L$ must be simple.

(3) If $B$ is a weak\ c-ideal of $L$ then there exists a subideal $C$ of 
$L$ such that%
\begin{equation*}
L=B+C\text{ and }B\cap C\leq B_{L}
\end{equation*}%
Then $K=K\cap L=K\cap \left( B+C\right) =B+\left( K\cap C\right) .$ Since $C$
is a subideal of $L$ there exists a chain of subalgebras%
\begin{equation*}
C=C_{0}<C_{1}<\ ...<C_{n}=L
\end{equation*}%
where $C_{j}$ is an ideal of $C_{j+1}$ for each $0\leq j\leq n-1.$ If we
intersect this chain with $K$ we get 
\begin{equation*}
C\cap K=C_{0}\cap K<C_{1}\cap K<\ ...<C_{n}\cap K=L\cap K=K
\end{equation*}%
and obviously $C_{j}\cap K$ is an ideal of $C_{j+1}\cap K$ for each $0\leq
j\leq n-1.$ Hence $C\cap K$ is a subideal of $K.$ Also,%
\begin{equation*}
B\cap \left( C\cap K\right) \leq B_{K}
\end{equation*}%
so that $B$ is a weak\ c-ideal of $L.$

(4) Suppose first that $B/I$ is a weak\ c-ideal of $L/ I.$ Then there
exists a subideal $C/ I$ of $L/ I$ such that%
\begin{equation*}
L/ I=B/ I+C/ I\text{ and }B/ I\cap C/ I\leq \left( B/ I\right) _{L/
I}=B_{L}/ I
\end{equation*}%
It follows that $L=B+C$ and $B\cap C\leq B_{L}$ where $C$ is a subideal of $%
L.$

Suppose conversely that $I$ is an ideal of $L$ with $I\leq B$ and $B$ is a 
weak\ c-ideal of $L.$ Then there exists a $C$ subideal of $L$ such that%
\begin{equation*}
L=B+C\text{ and }B\cap C\leq B_{L}.
\end{equation*}%
Since $I$ is an ideal and $I\leq B$ the factor algebra%
\begin{equation*}
L/ I=\left( B+C\right) / I=B/ I+\left( C+I\right) / I
\end{equation*}%
where $\left( C+I\right) / I$ is a subideal of $L/ I$ and%
\begin{equation*}
\left( B/ I\right) \cap \left( C+I\right) / I=(B\cap \left( C+I\right) )/
I=(I+B\cap C)/ I\leq B_{L}/ I=\left( B/ I\right) _{L/ I}
\end{equation*}%
so $B/ I$ is a $weak\ c$-$ideal$ of $L/ I.$
\end{proof}

The $Frattini\ subalgebra$\ of $L,\ F\left( L\right) ,\ $is the intersection
of all of the maximal subalgebras of $L.$ The $Frattini\ ideal,\ \phi (L),\ $%
of $L$ is $F\left( L\right) _{L}.$ The next result is a generalisation of
\cite[Proposition 2.2]{c-ideal}. The same proof works but we will include it for completeness.

\begin{proposition}
Let $B,C$ be subalgebras of $L$ with $B\leq F(C).$ If $B$ is a $weak\ c$-$%
ideal$ of $L$ then $B$ is an ideal of $L$ and $B\leq \phi (L).$
\end{proposition}
\begin{proof} Suppose that $L = B + K$ where $K$ is a subideal of $L$ and $B \cap K \leq B_L$. Then $C = C \cap L = C \cap (B + K) = B + C \cap K = C \cap K$ since $B \leq F(C)$. Hence $B \leq C \leq K$, giving $B = B \cap K \leq B_L$ and $B$ is an ideal of $L$. It then follows from \cite[Lemma 4.1]{[7]} that $B \leq \phi(L)$.
\end{proof}

An ideal $A$ is {\it complemented} in $L$ if there is a subalgebra $U$ of $L$ such that $L=A+U$
and $A\cap U=0.$  We adapt this to define a
complemented weak\ c-ideal as follows.

\begin{definition}
Let $L$ be a Lie algebra and $B$ is a $weak\ c$-$ideal$ of $L.$ A $weak\ c$-$%
ideal$ $B$ is complemented in $L$ if there is a subideal $C$ of $L$ such
that $L=B+C$ and $B\cap C=0.$
\end{definition}

Then we can give the following lemma:

\begin{lemma}
If $B$ is a $weak\ c$-$ideal$ of a Lie algebra $L$, then $B/ B_{L}$ has a
subideal complement in $L/ B_{L},\ i.e.,\ $there exists a \ subideal
subalgebra $C/ B_{L}$ of $L/ B_{L}$ such that $L/ B_{L}$is semidirect sum of 
$C/ B_{L}$ and $B/ B_{L}$. Conversely, if $B$ is a subalgebra of $L$ such
that $B/ B_{L}$ has a subideal complement in $L/ B_{L}$ then $B$ is a $weak\
c$-$ideal$ of $L $.
\end{lemma}

\begin{proof}
Let $B$ be a weak\ c-ideal of $L$. Then there exists a subideal $C$ of $%
L $ such that $B+C=L$ and $B\cap C\leq B_{L}.$ If $B_{L}=0$ then $B\cap C=0$
and so that $C$ is a subideal complement of $B$ in $L.$ Assume that $%
B_{L}\neq 0,$ then we can construct the factor algebras $B/ B_{L}$ and $%
\left( C+B_{L}\right) / B_{L}.$ If we intersect these two factor algebras we have
\begin{eqnarray*}
\frac{B}{ B_{L}} \cap \frac{C+B_{L}}{ B_{L}} &=&\frac{B\cap \left(
C+B_{L}\right)}{ B_{L}} \\
&=&\frac{B_{L}+(B\cap C)}{ B_{L}} \\
&= &\frac{B_{L}}{ B_{L}}=0
\end{eqnarray*}%
Hence, $(C+B_{L})/ B_{L}$ is a subideal complement of $B/ B_{L}$ in $L/
B_{L}.$ Conversely, if $K$ is a subideal of $L$ such that $K/ B_{L}$ is a
subideal complement of $B/ B_{L}$ in $L/ B_{L}$ then we have that%
\begin{equation*}
L/ B_{L}=\left( B/ B_{L}\right) +\left( K/ B_{L}\right) \text{ and }\left(
B/ B_{L}\right) \cap \left( K/ B_{L}\right) =0
\end{equation*}%
Then $L=B+K$ and $B\cap K\leq B_{L}.$ Therefore $B$ is a weak c-ideal
of $L.$
\end{proof}

\section{\textbf{SOME CHARACTERISATIONS OF SOLVABLE ALGEBRAS}}

We will use the following Lemma which is due to Stewart (\cite[Lemma 4.2.5]%
{stewart}).

\begin{lemma}
\label{l:stewart} Let $L$ be a Lie algebra over any field having two
subideals $H$ and $K$ such that $K$ is simple and not abelian. Suppose that $%
H\cap K=0$. Then $[H,K]=0$.
\end{lemma}

\begin{theorem}
Let $L$ be a Lie-algebra over a field $F$ of characteristic zero and let $B$
be an ideal of $L$. Then $B$ is solvable if and only if every maximal
subalgebra of $L$ not containing $B$ is a weak $c$-ideal of $L.$
\end{theorem}

\begin{proof}
Suppose every maximal subalgebra of $L$ not containing $B$ is a weak $c$%
-ideal of $L$. Then we need to show $B$ is solvable. Assume that this is
false and let $L$ be a minimal counter-example. Let $A$ be a minimal ideal
of $L$ and assume that $M/A$ is a maximal subalgebra of $L/A$ such that $%
(B+A)/A \not \subseteq M/A$. Then $M$ is a maximal subalgebra of $L$ with $B
\not \subseteq M$, so $M$ is a weak $c$-ideal of $L$. It follows that $M/A$
is a weak $c$-ideal of $L/A$, and hence that $(B+A)/A$ is solvable. If $%
B\cap A=0$, then $B\cong B/B\cap A \cong (B+A)/A$ is solvable. So we can
assume that every minmal ideal of $L$ is contained in $B$. Moreover, $B/A$
is solvable for each such minimal ideal. If $L$ has two distinct minimal
ideals $A_1$ and $A_2$ then $B\cong B/A_1\cap A_2$ is solvable, so $L$ is
monolithic with monolith $A$, say.

If $A$ is abelian then $B$ is solvable, so we must have that $A$ is simple.
Clearly, $B\not \subseteq \phi(L)$, since $\phi(L)$ is nilpotent, so there
is a maximal subalgebra $M$ of $L$ such that $B\not \subseteq M$. Then $M$
must be a weak $c$-ideal of $L$, so there is a subideal $C$ of $L$ such that 
$L=M+C$ and $M\cap C\subseteq M_L$. Since $B\not \subseteq M_L$ we have that 
$M_L=0$. It follows that $L$ is primitive of type $2$ and hence that $%
C_L(A)=0$, by \cite[Theorem 1.1]{prim}. But $[C,A]=0$ by Lemma \ref%
{l:stewart}, so $C=0$, a contradiction. Hence $B$ is solvable.

So suppose now that $B$ is solvable and let $M$ be a maximal ideal of $L$
not containing $B$. Then there exists $k\in {\mathbb{N}}$ such that $%
B^{(k+1)}\subseteq M$, but $B^{(k)}\not \subseteq M$. Clearly $L=M+B^{(k)}$
and $B^{(k)}\cap M$ is an ideal of $L$, so $B^{(k)}\cap M\subseteq M_L$. It
follows that $M$ is a $c$-ideal and hence a weak $c$-ideal of $L$.
\end{proof}

\begin{corollary}\label{c:max}
Let $L$ be a Lie algebra over a field $F$ of characteristic zero. Then $L$
is solvable if and only if every maximal subalgebra of $L$ is a weak $c$%
-ideal of $L$.
\end{corollary}

\bigskip

Unlike the corresponding results for c-ideals, the above two results do not hold in characteristic $p>0$, as the following
example shows.

\begin{example}
Let $L=sl(2)\otimes {\mathcal{O}}_1+1\otimes F(\frac{\partial}{\partial x}+x%
\frac{\partial}{\partial x})$, where ${\mathcal{O}}_1=F[x]$ with $x^p=0$ is
the truncated polynomial algebra in $1$ indeterminate and the ground field, $%
F$, is algebraically closed of characteristic $p>2$. Then $A=sl(2)\otimes {%
\mathcal{O}}_1$ is the unique minimal ideal of $L$. Put $%
S=sl(2)=Fu_{-1}+Fu_0+Fu_1$ with $[u_{-1},u_0]=u_{-1}$, $[u_{-1},u_1]=u_0$, $%
[u_0,u_1]=u_1$ and let $M=(F u_0+Fu_1) \otimes {\mathcal{O}}_1+1\otimes F(%
\frac{\partial}{\partial x}+x\frac{\partial}{\partial x})$. This is a
maximal subalgebra of $L$ which doesn't contain $A$. Suppose that it is a
weak $c$-ideal of $L$. Then there is a subideal $C$ of $L$ such that $L=C+M$
and $C\cap M\subseteq M_L=0$.

Let 
\begin{equation*}
C=C_{0}<C_{1}<\ ...<C_{n}=L
\end{equation*}%
where $C_{j}$ is an ideal of $C_{j+1}$ for each $0\leq j\leq n-1.$ Then $%
A\subseteq C_{n-1}$, so $A=C_{n-1}$ or $C_{n-1}=A+1\otimes F\frac{\partial}{%
\partial x}$. In the latter case it is straightforward to check that $%
C_{n-2}\subseteq A$. In either case, $C$ must be inside a proper ideal of $A$%
, and hence inside $S\oplus O_1^+$, where $O_1^+$ is spanned by $x, x^2,
\ldots, x^{p-1}$. But now $u_{-1}\otimes1 \not \in C+M$. Hence $M$ is not a
weak $c$-ideal of $L$.
\end{example}

\begin{lemma}
\label{l:solvsub} Let $L+U+C$ be a Lie algebra, where $U$ is a solvable
subalgebra of $L$ and $C$ is a subideal of $L$. Then there exists $n_0\in {%
\mathbb{N}}$ such that $L^{(n_0)}\subseteq C$.
\end{lemma}

\begin{proof}
Let $C=C_0< C_1< \ldots < C_k =L$ where $C_i$ is an ideal of $C_{i+1}$ for $%
0\leq i\leq k-1$. Then $L/C_{k-1}$ is solvable and so there exists $n_{k-1}$
such that $L^{(n_{k-1})}\subseteq C_{k-1}$. Suppose that $L^{(n_i)}\subseteq
C_i$ for some $0\leq i\leq k-1$. Now $C_i/C_{i-1}$ is solvable, and so there
is $r_i$ such that $C_i^{(r_i)}\subseteq C_{i-1}$. Hence $%
L^{(n_i+r_i)}=(L^{(n_i)})^{(r_i)}\subseteq C_{i-1}$. Put $n_{i-1}=n_i+r_i$.
The result now follows by induction.
\end{proof}

\begin{theorem}\label{t:solv}
Let $L$ be a Lie algebra over a field $F$ of characteristic zero. Then $L$ has a solvable maximal subalgebra that is a weak c-ideal of $L$ if and only if $L$ is solvable.
\end{theorem}
\begin{proof} Suppose first that $L$ has a solvable maximal subalgebra $M$ that is a weak c-ideal of $L$. We show that $L$ is solvable. Let $L$ be a minimal counter-example. Then there is a subideal $K$ of $L$ such that $L = M + K$ and $M \cap K \leq M_L$. If $M_L\neq 0$ then $L/M_L$ is solvable, by the minimality assumption, and $M_L$ is solvable, whence $L$ is solvable, a contradiction. It follows that $M_L=0$ and $L = M \dot{+} K$. If $R$ is the solvable radical of $L$ then $R \leq M_L = 0$, so $L$ is semisimple. But now, for all $n\geq 1$, $L=L^{(n)}\leq K \neq L$,by Lemma \ref{l:solvsub}, a contradiction. The result follows.
\par
The converse follows from Corollary \ref{c:max}.
\end{proof}

\begin{theorem}
Let $L$ be a Lie algebra over a field of characteristic zero such that all
maximal nilpotent subalgebras are weak $c$-ideals of $L$. Then $L$ is
solvable.
\end{theorem}

\begin{proof}
Suppose that $L$ is not solvable but that all maximal nilpotent subalgebras
of $L$ are weak $c$-ideals of $L$. Let $L=R\oplus S$ be the Levi
decomposition of $L$, where $S\neq 0$. Let $B$ be a maximal nilpotent
subalgebra of $S$ and $U$ be a maximal nilpotent subalgebra of $L$
containing it. Then there is a subideal $C$ of $L$ such that $L=U+C$ and $%
U\cap C\subseteq U_L$. It follows from Lemma \ref{l:solvsub} that $%
S=S^{(n_0)}\subseteq L^{(n_0)}\subseteq C$, and so $B\subseteq U\cap
C\subseteq U_L$, whence $S\cap U_L\neq 0$. But $S\cap U_L$ is an ideal of $S$
and so is semisimple. Since $U$ is nilpotent this is a contradiction.
\end{proof}

\begin{theorem}
\label{t:cart} Let $L$ be a Lie algebra, over a field $F$ of characteristic
zero, in which every Cartan subalgebra of $L$ is a weak c-ideal of $L$. Then 
$L$ is solvable.
\end{theorem}

\begin{proof}
Suppose that every Cartan subalgebra of $L$ is a weak c-ideal of $L$, and
that $L$ has a non-zero Levi factor $S$. Let $H$ be a Cartan subalgebra of $%
S $ and let $B$ be a Cartan subalgebra of its centralizer in the solvable
radical of $L$. Then $C = H + B$ is a Cartan subalgebra of $L$ (see \cite%
{dix}) and there is a subideal $K$ of $L$ such that $L = C + K$ and $C \cap
K \leq C_L$. Now there is an $r \geq 2$ such that $L^{(r)} \leq K$, by Lemma %
\ref{l:solvsub}. But $S \leq L^{(r)} \leq K$, so $C \cap S \leq C \cap K
\leq C_L$ giving $C \cap S \leq C_L \cap S = 0$, a contradiction. It follows
that $S = 0$ and hence that $L$ is solvable.
\end{proof}

\section{\textbf{SOME CHARACTERISATIONS OF SUPERSOLVABLE ALGEBRAS}}

The following is proved in \cite[Lemma 4.1]{c-ideal}

\begin{lemma}
\label{l:maxnilp} Let $L$ be a Lie algebra over any field $F$, let $A$ be an
ideal of $L$ and let $U/A$ be a maximal nilpotent subalgebra of $L/A$. Then $%
U = C+A$, where $C$ is a maximal nilpotent subalgebra of $L$.
\end{lemma}

We will also need the following  result.

\begin{lemma}\label{l:min}  Let $L$ be a Lie algebra over any field $F$ and suppose that $L=B+K$,
where $B$ is a nilpotent subalgebra and $K$ is a subideal of $L$. Then there exists $s\in {\mathbb {N}}$
such that $L^s\subseteq K$. Moreover, if $A$ is a minimal ideal of $L$ then either $A\subseteq K$ or
$[L,A]=0$.
\end{lemma}
\begin{proof}  Since $K$ is a subideal of $L$, there exists $
r\in {\mathbb{N}}$ such that $L$\thinspace (ad $K)^{r}\subseteq K$.
As $B$ is nilpotent, there exists $s\in {\mathbb{N}}$ such that $
L^{s}=(B+K)^{s}\subseteq K$. 
\par

Now $[L,A]=A$ or $[L,A]=0$ and the former
implies that $A\subseteq L^{s}\subseteq K$.
\end{proof}

\begin{lemma}\label{l:fac}
\label{l:reduce} Let $L$ be a Lie algebra, over any field $F$, in which
every maximal subalgebra of each maximal nilpotent subalgebra of $L$ is a
weak c-ideal of $L$, and let $A$ be a minimal abelian ideal of $L$. Then
every maximal subalgebra of each maximal nilpotent subalgebra of $L/A$ is a
weak c-ideal of $L/A$.
\end{lemma}

\begin{proof}
Suppose that $U/A$ is a maximal nilpotent subalgebra of $L/A$. Then $U = C +
A$ where $C$ is a maximal nilpotent subalgebra of $L$ by Lemma \ref%
{l:maxnilp}. Let $B/A$ be a maximal subalgebra of $U/A$. Then $B = B \cap
(C+A) = B \cap C + A = D + A$ where $D$ is a maximal subalgebra of $C$ with $%
B \cap C \leq D$. Now $D$ is a weak c-ideal of $L$ so there is a subideal $K$
of $L$ with $L = D+K$ and $D \cap K \leq D_L$.

If $A \leq K$ we have $$\frac{L}{A} = \frac{D+K}{A }= \frac{D+A}{A} + \frac{K}{A} = \frac{B}{A} + \frac{K}{A},$$
and $$\frac{B}{A} \cap \frac{K}{A} = \frac{B \cap K}{A} = \frac{(D+A) \cap K}{A} = \frac{D \cap K + A}{A}
\leq \frac{D_L + A}{A} \leq \left(\frac{B}{A}\right)_{L/A}.$$

So suppose that $A\not\leq K$. Then Lemma \ref{l:min} shows that $[L,A]=0$.
It follows that $A\leq C$ and $B=D$. We have $L=B+K$ and $B\cap K\leq B_{L}$%
, so $$\frac{L}{A}=\frac{B}{A}+\frac{K+A}{A}$$ and $$\frac{B}{A}\cap \frac{K+A}{A}=\frac{B\cap (K+A)}{A}=\frac{B\cap
K+A}{A}\leq \frac{B_{L}+A}{A}\leq \left(\frac{B}{A}\right)_{L/A}.$$
\end{proof}

\begin{lemma}\label{l:one}
 Let $L$ be a Lie algebra over any field $F$, in which every maximal nilpotent subalgebra of $L$ is a
weak c-ideal of $L,$ and suppose that $A$ is a minimal abelian ideal of $L$
and $M$ is a core-free maximal subalgebra of $L.$ Then $A$ is one
dimensional.
\end{lemma}

\begin{proof}
We have that $L=A\dot{+} M$ and $A$ is the unique minimal ideal of $L,$ by
\cite[Theorem 1.1]{prim}. Let $C$ be a maximal nilpotent
subalgebra of $L$ with $A\leq C.$ If $A=C,$ choose $B$ to be a maximal
subalgebra of $A,$ so that $A=B+Fa$ and $B_{L}=0.$ Then $B$ is a weak
c-ideal of $L.$ So there is a subideal of $K$ of $L$ with $L=B+K$ and $B\cap
K\leq B_{L}=0$. Now $L=B+K=B+K^L=K^L$, since $B\leq A\leq K^L$. It follows that
$K=L$, whence $B=0$ and $A=Fa$ is one dimensional.
\par

So suppose
that $C\neq A.$ Then $C=A+M\cap C.$ Let $B$ be a maximal subalgebra of $C$
containing $M\cap C.$ Then $B$ is a weak c-ideal of $L,$ so there is a
subideal $K$ of $L$ with $L=B+K$ and $B\cap K\leq B_{L}.$ If $A\leq
B_{L}\leq B,$ we have $C=A+M\cap C\leq B$, a contradiction. Hence $B_{L}=0$
and $L=B\dot{+} K.$ Now $C=B+C\cap K$ and $B\cap C\cap K=B\cap K=0.$ As $C$
is nilpotent this means that $\dim (C\cap K)=1.$ If $A \subseteq K$ we have that $A\leq C\cap K,$ so $
\dim A=1,$ as required. Otherwise, $[L,A]=0$, by Lemma \ref{l:min}, and again $\dim A=1$.
\end{proof}

We can now prove our main result.

\begin{theorem}\label{t:sup}
Let $L$ be a solvable Lie algebra over any field $F$ in which every maximal subalgebra of each maximal
nilpotent subalgebra of $L$ is a weak c-ideal of $L.$ Then $L$ is
supersolvable.
\end{theorem}

\begin{proof}
Let $L$ be a minimal counter-example and let $A$ be a minimal abelian ideal
of $L.$ Then $L/A$ satisfies the same hypothesis by Lemma \ref{l:fac}. We thus have
that $L/A$ is supersolvable and it remains to show that $\dim A=1.$

If there is anaother minimal ideal $I$ of $L,$ then 
\begin{equation*}
A\cong \left( A+I\right) /I\leq L/I\text{ }
\end{equation*}
which is supersolvable and so $\dim A=1.$ So we can assume that $A$ is the
unique minimal ideal of $L.$ Also, if $A\leq \phi \left( L\right) ,$ we have
that $L/\phi (L)$ is supersolvable, whence $L$ is supersolvable by
\cite[Theorem 7]{Barnes}. We therefore, further assume that $A\nleq \phi (L).
$ It follows that $L=A\dot{+} M,$ where $M$ is a core-free maximal subalgebra
of $L.$ The result now follows from Lemma \ref{l:one}.
\end{proof}

If $L$ has no one-dimensional maximal nilpotent subalgebras, we can remove
the solvability assumption from the above result provided that $F$ has characteristic zero.

\begin{corollary}\label{c:sup}
Let $L$ be a Lie algebra over a field $F$ of characteristic zero in which
every maximal nilpotent subalgebra has dimension at least two. If every
maximal subalgebra of each maximal nilpotent subalgebra of $L$ is a weak
c-ideal of $L,$ then $L$ is supersolvable.
\end{corollary}

\begin{proof}
Let $N$ be the nilradical of $L,$ and let $x\notin N.$ Then $x\in C$ for
some maximal nilpotent subalgebra $C$ of $L.$ Since $\dim C>1$,
there is a maximal subalgebra $B$ of $C$ with $x\in B.$ Then there is a
subideal $K$ of $L$ such that $L=B+K$ and $B\cap K\subseteq B_{L}\leq
C_{L}\leq N$. Clearly, $x\notin K,$ since otherwise $x\in B\cap K\leq N.$
Moreover, $L^r\subseteq K$ for some $r\in {\mathbb{N}}$, by Lemma \ref{l:min}. 
We have shown that if $x\notin N$ there is a
subideal $K$ of $L$ with $x\notin K$ and $L^r\subseteq K$. 
par

Suppose that $L$ is not solvable. Then there is a semisimple Levi factor $S$ of $L$. Choose
$x\in S$. Then $x\in S = S^r\subseteq K$, a contradiction. Thus $L$ is solvable and the result follows from
Theorem \ref{t:sup}.
\end{proof}

If $L$ has a one-dimensional maximal nilpotent subalgebra, then we can also
remove the solvability assumption from Theorem 4.4., provided that
underlying field $F$ has again characteristic zero and $L$ is not three-dimensional simple.

\begin{corollary}
Let $L$ be a Lie algebra over a field $F$ of characteristic zero. If every
maximal subalgebra of each maximal nilpotent subalgebra of $L$ is a weak $c$%
-ideal of $L,$ then $L$ is supersolvable or three dimensional simple.
\end{corollary}

\begin{proof}
If every maximal nilpotent subalgebra of $L$ has dimension at least two,
then $L$ is supersolvable by Corollary \ref{c:sup}. So we need only consider the
case where $L$ has a one-dimensional maximal nilpotent subalgebra say $Fx$.
\par

Suppose first that $L$ is semisimple, so $L=S_{1}\oplus ...\oplus S_{n},$
where $S_{i}$ is a simple ideal of $L$ for $1\leq i\leq n.$ Let $n>1.$ If $
x\in S_{i},$ then choosing $s\in S_{j}$ with $j\neq i,$ we have that $Fx+Fs$
is a two dimensional abelian subalgebra, which contradicts the maximality of 
$Fx.$ If $x\notin S_{i},$ for every $1\leq i\leq n,$ then $x$ has nonzero
projections in at least two of the $S_{k}$'s, say $s_{i}\in S_{i}$ and $
s_{j}\in S_{j}.$ But then $Fx+Fs_{i}$ is a two-dimensional abelian
subalgebra, a contradiction again. It follows that $L$ is simple. But then $Fx$ is a Cartan subalgebra of $L,$
which yields that $L$ has rank one and thus is three dimensional.

So now let $L$ be a minimal-counter example. We have seen that $L$ is not
semisimple, so it has a minimal abelian ideal $A.$ By Lemma \ref{l:fac}, $L/A$ is supersolvable
or three-dimensional simple. In the former case, $L$ is solvable and so is supersolvable, by Theorem \ref{t:sup}.
\par

In the latter case, $L=A\oplus S$ where $S$ is three-dimensional simple, and
so a core-free maximal subalgebra of $L.$ It follows from Lemma \ref{l:one} that $
\dim A=1.$ But now $C_{L}\left( A\right) =A$ or $L.$ In the former case $%
S\cong L/A=L/C_{L}\left( A\right) \cong Inn\left( A\right) ,$ a subalgebra
of $Der\left( A\right) ,$ which is impo\i ssible. Hence $L=A\oplus S,$ where 
$A$ and $S$ are both ideals of $L$ and again $L$ has no one-dimensional
maximal nilpotent subalgebras.
\end{proof}

\section{\textbf{ONE-DIMENSIONAL WEAK C-IDEALS}}

\begin{lemma}\label{l:equiv}
Let $L$ be a Lie algebra over any field $F$. Then the
one-dimensional subalgebra $Fx$ of $L$ is a weak c-ideal of $L$ if and only
if it is a $c$-ideal of $L$
\end{lemma}

\begin{proof}
Let $Fx$ be a weak c-ideal of $L.$ Then there is a subideal $K$ of $L$ such
that $L=Fx+K$ and $Fx\cap K\leq (Fx) _L. $ Since either $K=L$ or $K$ has codimension one in $L$, it is an ideal of $L$
and $Fx$ is a $c$-ideal of $L$.
\end{proof}

We say that $L$ is {\it almost abelian} if $L=L^2\oplus Fx$, where $L^2$ is abelian and
$[x,y]=y$ for all $y\in L^2$. Then the following result follows from Lemma \ref{l:equiv}
and \cite[Theorem 5.2]{c-ideal}.

\begin{theorem}
Let $L$ be a Lie algebra over any field $F$. Then all
one-dimensional subalgebras of $L$ are weak c-ideals of $L$ if and only if:

$(i)$ $L^{3}=0;$ or

$(ii)\ L=A\oplus B,$ where $A$ is an abelian ideal of $L$ and $B$ is an
almost abelian ideal of $L.$
\end{theorem}

\end{document}